\documentclass[a4paper,11pt,reqno]{amsart}
\usepackage{amssymb, amsmath, graphicx, amsthm, enumitem, multicol,amsfonts}
\usepackage[margin=1in]{geometry}
\usepackage{cite}
\usepackage{ bbold }
\usepackage{longtable}
\usepackage{booktabs}
\usepackage{placeins}
\usepackage{hyperref}
\usepackage{tikz-cd}
\usepackage[utf8]{inputenc}
\theoremstyle{definition}
\newtheorem{definition}{Definition}
\newtheorem {proposition}   {Proposition}
\newtheorem {theorem}    {Theorem}

\DeclareMathOperator{\C}{\mathbb{C}}

\DeclareMathOperator{\Res}{Res}
\DeclareMathOperator{\Z}{\mathbb{Z}}
\DeclareMathOperator{\N}{\mathbb{N}}

\DeclareMathOperator{\End}{End}

\newtheorem {corollary}  [theorem]    {Corollary}

\newtheorem*{remark}{Remark}
\newcommand{\ba}{\begin{align}}
\newcommand{\ea}{\end{align}}
\newcommand{\bea}{\begin{eqnarray}}
\newcommand{\eea}{\end{eqnarray}}
\newcommand{\be}{\begin{equation}}
\newcommand{\ee}{\end{equation}}

\numberwithin{equation}{subsection}

\theoremstyle{plain}

\newtheorem*{theorem-non}{Theorem}

\theoremstyle{definition}

\theoremstyle{remark}

\author{Darlayne Addabbo}
\email{addabbd@sunypoly.edu}

\title{The Huang Algebra Ideal and the Diagonal Shift Property}
\date{\today}

\begin{document}


\maketitle
\begin{abstract}Let $V$ be a grading-restricted vertex algebra and let $A^\infty(V )=U^\infty(V)/Q^\infty(V)$ 
be the associative algebra constructed by Huang, where $U^\infty(V)$ is the space of 
column-finite infinite matrices with entries in $V$ and $Q^\infty(V)$ is an ideal of 
a (nonassociative) algebra structure on $U^\infty(V)$ defined by Huang. Huang introduced families of elements in $Q^\infty(V)$ and conjectured that these elements generate $Q^\infty(V)$. We introduce the ``diagonal shift property" and prove that the elements that Huang conjectured as a generating set for $Q^\infty(V)$ satisfy this property. On the other hand, in the case that $V$ is the rank one Heisenberg vertex operator algebra, we construct infinitely many linearly independent elements in $Q^\infty(V)$ that do not satisfy the diagonal shift property. As a corollary, we disprove Huang's conjecture.
\end{abstract}
\section{Introduction}
Given a vertex operator algebra (VOA) $V$ Zhu defined an associative algebra $A(V)$ such that there is a one-to-one correspondence between irreducible admissible $V$-modules and $A(V)$-modules \cite{Z}. Frenkel and Zhu proved that in the case that $V$ is rational, the category of admissible $V$-modules and the category of $A(V)$-modules are equivalent \cite{FZ}. These algebras were crucial in Zhu's proof of the modular invariance of correlation functions for elements in rational $C_2$-cofinite VOAs.

While Zhu algebras encode the essential information about modules of rational VOAs, if one wants to study irrational VOAs, Zhu algebras are in general insufficient. For this reason, Dong, Li, and Mason introduced ``higher level Zhu algebras". Given a VOA $V$ one has an infinite family of associative algebras $A_n(V)$, $n\in \mathbb{N}$. The level zero Zhu algebra $A_0(V)$ is Zhu's original algebra $A(V).$ Dong, Li, and Mason proved several results about the relationship between $V$-modules and $A_n(V)$-modules, and this relationship was further studied in \cite{BVY_gen}. Miyamoto introduced pseudo-traces for $C_2$-cofinite irrational VOAs and proved that the space of these pseudo-traces is invariant under the action of $\mathrm{SL}_2(\Z)$ \cite{Miy}.

In \cite{H}, Huang proved the modular invariance of intertwining operators for rational VOAs satisfying the same conditions as those in \cite{Z}. He subsequently conjectured the modular invariance of logarithmic intertwining operators for $C_2$-cofinite, CFT-type VOAs.

In \cite{H2}, given a grading-restricted vertex algebra, Huang introduced an associative algebra $A^\infty(V)$ and proved that the set of equivalence classes of lower-bounded generalized $V$-modules are in bijection with the set of equivalence classes of nondegenerate graded $A^\infty(V)$-modules. He also defined a family of subalgebras $A^n(V)$, $n\in \mathbb{N}$, of $A^\infty(V)$ and proved that there is a bijection between the set of equivalence classes of irreducible lower-bounded generalized $V$-modules and the set of equivalence classes of irreducible nondegenerate graded $A^n(V)$-modules. Given $n\in \mathbb{N}$, the higher level Zhu algebra $A_n(V)$ is a subalgebra of $A^n(V)$, and hence $A^\infty(V)$ contains all of the higher level Zhu algebras as subalgebras. Huang further studied the algebras $A^\infty(V)$ and $A^n(V)$, $n\in \N$, in \cite{H3}, and in \cite{H5} he used them to prove his conjecture on the modular invariance of logarithmic intertwining operators, a result which subsumed the earlier modular invariance results of Huang, Miyamoto, and Zhu, and filled some gaps in \cite{Miy}.

Given a VOA $V$ and $n\in \mathbb{N}$, the higher level Zhu algebra $A_n(V)$ is a quotient of $V$ by an ideal denoted $O_n(V).$ One has an explicit description of $O_n(V)$ in terms of VOA elements \cite{DLM}, which we review below in Section \ref{higher_level}. This explicit description is useful in describing higher level Zhu algebras in terms of generators and relations. It allows one to construct higher level Zhu algebras using minimal information about the modules of the VOAs at hand, without relying on information about the modules that the higher level Zhu algebras encode, see e.g. \cite{AB1, AB2, BVY_Heis, BVY_Vir}. This is crucial if one wants to use higher level Zhu algebras to study the modules of VOAs.

Given a VOA $V$, $n\in \mathbb{N},$ and an admissible $V$-module $M$, Dong, Li, and Mason \cite{DLM} defined \begin{equation*}
\Omega_n(M):=\{w\in M|v_kw=0\text{ if wt}\,v-k-1<-n\text{ for }v\in V\text{ homogeneous weight}\}.
\end{equation*} The subspace $\Omega_n(M)$ of $M$ is an $A_n(V)$ module and the action of $A_n(V)$ is given by zero modes: Any element $v+O_n(V)$ in $A_n(V)$ acts on $\Omega_n(M)$ by the zero mode $o(v)$ of $v$. The ideal $O_n(V)$ consists precisely of every element in $V$ whose zero mode acts as $0$ on $\Omega_n(M)$ for all admissible $V$-modules $M$ \cite{DLM}.

In contrast to the higher level Zhu algebras which encode information about the action of VOAs on their admissible modules by zero modes, Huang algebras encode the action by all modes. Given a grading-restricted vertex algebra, the Huang algebra is a quotient not of $V$ but of the space $U^\infty(V)$ of column-finite infinite matrices, doubly indexed by $\N$, with entries in $V$. The associative algebra $A^\infty(V)$ is then defined to be $U^\infty(V)/Q^\infty(V),$ where $Q^\infty(V)$ is a certain ideal that is defined abstractly: Given a lower-bounded generalized $V$-module $M$, Huang constructs an $\N$-graded space $\mathrm{Gr}(M)=\displaystyle\oplus_{n\in \mathbb{N}}\,\mathrm{Gr}_n(M)$ and defines an action of $U^\infty(V)$ on this space. Then $Q^\infty(V)$ consists of all elements of $U^\infty(V)$ that act as zero on $\mathrm{Gr}(M)$ for all lower-bounded generalized $V$-modules $M$ \cite{H2}.

To date, there is no explicit description of the ideal $Q^\infty(V)$ in terms of vertex algebra elements. In \cite{H2}, given a grading-restricted vertex algebra $V$, Huang defined a subspace $O^\infty(V)$ of $U^\infty(V)$ analogous to the subspaces $O_n(V)$ of $V$ and proved that $O^\infty(V)$ is a subspace of $Q^\infty(V).$ He also defined elements of $U^\infty(V)$ arising from the Jacobi identity and proved that these elements are in $Q^\infty(V)$ \cite{H3}. He conjectured in \cite{H2} that $O^\infty(V)$ and these Jacobi identity elements generate all of $Q^\infty(V)$. In this paper, we introduce a property called the ``diagonal shift property" and prove that elements of $O^\infty(V)$ and the Jacobi identity elements constructed by Huang satisfy this property. In the case that $V$ is the rank one Heisenberg VOA we prove that there exist elements of $Q^\infty(V)$ that do not satisfy the diagonal shift property, and as a corollary, we find that Huang's conjecture is not in general true.

We briefly summarize the structure of this paper. In Section \ref{background} we review the definitions of grading-restricted vertex algebras and their lower-bounded generalized $V$-modules. We then review higher level Zhu algebras and some facts about them which will be used in the paper. Following this, we provide background on Huang algebras. In Section \ref{diagonal shift} we define the diagonal shift property and prove that for a grading-restricted vertex algebra, elements of $O^\infty(V)$ and the Jacobi identity elements constructed by Huang satisfy the diagonal shift property. In Section \ref{Heis} we review the construction of the rank one Heisenberg VOA $M(1)$ and construct an infinite family of linearly independent elements in $Q^\infty(M(1))$ that do not satisfy the diagonal shift property. All vector spaces in this paper are assumed to be complex, $\N$ denotes the nonnegative integers, and $\Z_+$ denotes the strictly positive integers. 

\subsection*{Acknowledgments}
The author is extremely grateful to Yi-Zhi Huang for helpful comments on this work. She is also grateful for an AMS-Simons Research Enhancement PUI Grant 104909. This work was partially completed during the author's stay at the Max Planck Institute for Mathematics. She thanks MPIM for their hospitality. She is also grateful to an anonymous referee for careful reading of the manuscript and helpful suggestions.

\section{Background}\label{background}

Huang algebras are constructed for grading-restricted vertex algebras, which are generalizations of vertex operator algebras. For background on vertex operator algebras, we refer the reader to the standard texts \cite{FLM, FHL, LL}. The definition of a grading-restricted vertex algebra $V$ can be found e.g. in \cite{H1}. To keep this paper self-contained, we restate the definition here. While the definition in \cite{H1} uses the duality properties rather than explicitly stating the Jacobi identity, here we state the definition using the Jacobi identity, as this will be used often throughout the paper. For the same reason, the definition we give of a lower-bounded generalized $V$-module will involve an explicit statement of the Jacobi identity rather than the duality properties.
\subsection{Grading-restricted vertex algebras and their lower-bounded generalized modules}
\begin{definition}\cite{H1} A \emph{grading-restricted vertex algebra} is a $\Z$-graded vector space $V=\coprod_{n\in\Z}V_{(n)}$ equipped with a vertex operator map 
\begin{eqnarray*}
V&\rightarrow& \End\, V[[x, x^{-1}]],\\v&\mapsto& Y(v, x)=\sum_{n\in \Z}v_n x^{-n-1}
\end{eqnarray*}and a distinguished vector ${\bf 1}\in V_{(0)}$ called the vacuum vector, satisfying the following conditions for all $a,b\in V$:
\begin{enumerate}[label=\Roman*.]
\item Grading restriction condition: For $n\in\Z$, $\dim V_{(n)}<\infty$ and $V_{(n)}=0$ when $n$ is sufficiently negative.
\item Lower truncation condition: $a_nb=0$ for $n$ sufficiently large.
\item Identity property: $Y({\bf 1},x)=1.$
\item Creation property: $Y(a,x){\bf 1}\in V[[x]]$ and $\lim_{x\rightarrow 0}Y(a,x){\bf 1}=a.$
\item The Jacobi identity:
\begin{equation}\sum_{i=0}^\infty\binom{m}{i}(a_{l+i}b)_{m+n-i}=\sum_{i=0}^\infty(-1)^i\binom{l}{i}a_{m+l-i}b_{n+i}-\sum_{i=0}^\infty(-1)^{l+i}\binom{l}{i}b_{n+l-i}a_{m+i}
\end{equation}for all $l,m,n\in\Z.$
\item $L(0)$-bracket formula: Let $L(0):V\rightarrow V$ be defined by $L(0)v=nv$ for $v\in V_{(n)}$. Then
\begin{equation}\label{Lzero}
[L(0), Y(a,x)]=Y(L(0)a,x)+x\frac{d}{dx}Y(a,x).
\end{equation}
\item $L(-1)$-derivative formula: Let $L(-1):V\rightarrow V$ be the operator defined by
\begin{equation}\label{Lminus}L(-1)v=\Res_x x^{-2}Y(v,x){\bf 1}
\end{equation} for $v\in V.$ Then for $v\in V$ we have
\begin{equation}\frac{d}{dx}Y(u,x)=Y(L(-1)u,x)=[L(-1),Y(u,x)].
\end{equation}
\end{enumerate}
\end{definition}
Essentially, a grading-restricted vertex algebra satisfies all of the properties of a vertex operator algebra, except that instead of requiring the existence of a conformal vector, we require that items VI. and VII. above hold.
\begin{definition}\cite{H2, H4} Given a grading-restricted vertex algebra $V$, a \emph{lower-bounded generalized $V$-module} is a $\C$-graded vector space $W=\coprod_{\lambda\in \C}W_{(\lambda)}$ with a vertex operator map
\begin{eqnarray*}
V&\rightarrow& \End\, W[[x, x^{-1}]],\\v&\mapsto& Y_W(v, x)=\sum_{n\in \Z}v_n x^{-n-1}
\end{eqnarray*} satisfying
\begin{enumerate}[label=\roman*.]
\item Grading restriction condition: $W_{(\lambda)}=0$ when the real part of $\lambda$ is sufficiently negative. In contrast to the grading restriction condition for a grading-restricted vertex algebra, we do not require that $\dim W_{(\lambda)}<\infty.$
\item Lower truncation condition: For all $a\in V$ and $w\in W$, $a_nw=0$ for $n\in \Z$ sufficiently large.
\item Identity property: $Y_W({\bf 1},x)=1.$
\item The Jacobi identity:
\begin{equation}\sum_{i=0}^\infty\binom{m}{i}(a_{l+i}b)_{m+n-i}=\sum_{i=0}^\infty(-1)^i\binom{l}{i}a_{m+l-i}b_{n+i}-\sum_{i=0}^\infty(-1)^{l+i}\binom{l}{i}b_{n+l-i}a_{m+i}
\end{equation}acting on $W$, for all $l,m,n\in\Z$ and $a,b\in V.$
\item $L_W(0)$-bracket formula: There exists an operator $L_W(0):W\rightarrow W$ such that \begin{equation}[L_W(0),Y_W(a,x)]=Y_W(L(0)a,x)+x\frac{d}{dx}Y_W(a,x).
\end{equation}In contrast to grading-restricted vertex algebras, the $W_{(\lambda)}$ are no longer eigenspaces of $L_W(0)$ with eigenvalue $n$ but instead are generalized eigenspaces with eigenvalue $\lambda$: $W_{(\lambda)}:=\{w\in W|(L_W(0)-\lambda)^kw=0\text{ for some }k\in \Z_+\}.$
\item $L_W(-1)$-derivative formula: There is an operator $L_W(-1):W\rightarrow W$ such that for all $u\in V$
\begin{equation}\frac{d}{dx}Y_W(u,x)=Y_W(L(-1)u,x)=[L_W(-1),Y_W(u,x)].
\end{equation}
\end{enumerate}
 
\end{definition}

Lower-bounded generalized modules of a grading-restricted vertex algebra have a canonical $\N$-grading. The Huang algebra $A^\infty(V)$ can be defined using the canonical $\N$-grading of $V$-modules \cite{H3}, however we will not use this description in this paper. We review the definition of the canonical $\N$-grading here since it will be used in Section \ref{Heis}.

\begin{definition} \cite{H3}
Let $V$ be a grading-restricted vertex algebra and let $W$ be a lower-bounded generalized $V$-module. For $\mu\in \mathbb{C}/\mathbb{Z}$, let $W^\mu$ be the generalized $V$-submodule of $W$ spanned by homogeneous elements of weights in $\mu$. Define
\[\Gamma(W)=\{\mu \in \mathbb{C}/\mathbb{Z}|W^\mu\ne 0\}.\] For $\mu\in \Gamma(W)$, there exists $h^\mu\in \mathbb{C}$ such that 
\[W^\mu=\coprod_{n\in \mathbb{N}}W_{[h^\mu+n]}\] and $W_{[h^\mu]}\ne 0$.  Then 
\[W=\coprod_{\mu\in \Gamma(W)}W^\mu=\coprod_{\mu \in \Gamma(W)}\coprod_{n\in \mathbb{N}}W_{[h^\mu+n]}.\] Given $n\in \mathbb{N}$, define
\[W_{[[n]]}=\coprod_{\mu \in \Gamma(W)}W_{[h^\mu+n]}.\] Then the \emph{canonical $\mathbb{N}$-grading} of $W$ is given by
\[W=\coprod_{n\in \mathbb{N}}W_{[[n]]}.\]
\end{definition}

\subsection{Higher Level Zhu Algebras}\label{higher_level}
In this subsection we review higher level Zhu algebras and facts about them which will motivate some of the proofs in this paper. This material is from \cite{DLM}, see also \cite{AB1} and \cite{BVY_gen}.

\begin{definition}\cite{DLM} Given a vertex operator algebra $V$ and $n\in \mathbb{N}$, define
$O_n(V)$ to be the subspace of $V$ spanned by elements of the form
\begin{equation}\label{define-circ}
u \circ_n v =
\Res_x \frac{(1 + x)^{\mathrm{wt}\, u + n}Y(u, x)v}{x^{2n+2}}
\end{equation}
for all homogeneous $u \in V$ and for all $v \in V$, and by elements of the form $(L(-1) + L(0))v$ for all $v \in V$. The vector space $A_n(V)$ is defined to be the quotient space $V/O_n(V)$. Then $A_n(V)$ is an associative algebra with product $*_n$ defined by
\begin{equation}\label{*_n-definition}
u *_n v = \sum_{m=0}^n(-1)^m\binom{m+n}{n}\Res_x \frac{(1 + x)^{\mathrm{wt}\, u + n}Y(u, x)v}{x^{n+m+1}},
\end{equation}
for $v \in V$ and homogeneous $u \in V$. For general $u \in V$, $*_n$ is defined by linearity. The algebra $A_n(V)$ is called \emph{the level $n$ Zhu algebra of $V$}.
\end{definition}

By Lemma 2.1 (i) in \cite{DLM}, we have \begin{proposition}\label{prop1}
Let $V$ be a VOA and $u\in V$ homogeneous, $v\in V$, and $m\ge k\ge 0$. Then 
\begin{equation}
\Res_x Y(u, x)v\frac{(1+x)^{\mathrm{wt}\, u+n+k}}{x^{2n+2+m}}\in O_n(V).
\end{equation}
\end{proposition}

As is stated in the proof of Proposition 2.4 in \cite{DLM}, an immediate consequence of Proposition \ref{prop1} is the following corollary.
\begin{corollary}\label{corZhu}Let $V$ be a VOA and $n\in \mathbb{N}$, $n>0$. Then \begin{equation}
O_{n}(V)\subset O_{n-1}(V).
\end{equation}
\end{corollary}

We also recall from the proof of Proposition 2.4 in \cite{DLM}:
\begin{proposition}\label{mult_equiv}
Let $V$ be a VOA and let $n\in \N$, $n>0$. Then given $u,v\in V$ \begin{equation}u*_{n}v\equiv u*_{n-1}v\mod O_{n-1}(V).\end{equation}
\end{proposition}

Corollary \ref{corZhu} and Proposition \ref{mult_equiv} together imply the following result, which is Proposition 2.4 in \cite{DLM}: 
\begin{proposition}Let $V$ be a VOA and $n\in \N$, $n>0$. Then the map from $A_n(V)$ to $A_{n-1}(V)$ given by $v+O_{n}(V)\mapsto v+O_{n-1}(V)$ is a surjective algebra homomorphism. \end{proposition}
\subsection{Huang Algebras}

We next review the construction of the Huang algebra $A^\infty(V)$ of a grading-restricted vertex algebra $V$. This material is from \cite{H3} and \cite{H2} but we restate it here for the convenience of the reader.

Define $U^\infty(\mathbb{C})$ to be the space of all column-finite infinite matrices with entries in $\mathbb{C}$, doubly indexed by $\N$. In other words, $U^\infty(\mathbb{C})$ consists of all matrices $[a_{kl}]$ where $a_{kl}\in \mathbb{C}$, $k, l\in \mathbb{N}$ and for each fixed $l\in \mathbb{N}$, only finitely many $a_{kl}$ are nonzero. Elements of $U^\infty(\mathbb{C})$ are infinite linear combinations of the elementary matrices $E_{kl}$, where $E_{kl}$ denotes the matrix with all entries equal to $0$ except for the $k,l$ position whose entry is $1$. Given $A=\sum_{k,l\in \mathbb{N}}a_{kl}E_{kl}$ and $B=\sum_{k,l\in \mathbb{N}}b_{kl}E_{kl}$, $a_{kl},b_{kl}\in \C,$
\begin{equation}
AB=\sum_{k,n\in \mathbb{N}}\big(\sum_{m\in \mathbb{N}}a_{km}b_{mn}\big)E_{kn}.
\end{equation}The column-finiteness requirement ensures that this is a well-defined product.

Now, let $U^\infty(V)=V\otimes U^\infty(\mathbb{C})$. Denote $v\otimes E_{kl}$ by $[v]_{kl}$.

Next, define a product $\diamond$ on $U^\infty(V)$ as follows: Given $u, v\in V$ and $k, l, m, n\in \mathbb{N}$, if $m\ne n,$
\begin{equation*}
[u]_{km}\diamond[v]_{nl}=0
\end{equation*}and if $m=n,$
\begin{equation}\label{diamond}
[u]_{kn}\diamond[v]_{nl}=\sum_{m=0}^n\binom{-k+n-l-1}{m}\Res_xx^{-k+n-l-m-1}(1+x)^l[Y((1+x)^{L(0)}u, x)v]_{kl}. 
\end{equation}Extend this product linearly to a product on $U^\infty(V)$.

\begin{remark}When $k=n=l$, the product $\diamond$ coincides with the product $*_n$ defined in \cite{DLM}. That is, $[u]_{nn}\diamond[v]_{nn}=[u*_nv]_{nn}.$
\end{remark}

We next review, given a grading-restricted vertex algebra $V$, the associated graded space $\mathrm{Gr}(W)$ of an ascending filtration of a lower-bounded generalized $V$-module $W$ \cite{H2}.

Let $W$ be a lower-bounded generalized $V$-module. For $n\in \mathbb{N}$, define
\begin{eqnarray*}
\Omega_n(W)=\{w\in W|v_kw=0{\text{ for homogeneous }v\in V{\text{ and }\text{wt }v-k-1<-n}}\}. 
\end{eqnarray*}We have $\Omega_n\subset \Omega_m$ for $n\le m$ and since $W$ is a lower-bounded generalized $V$-module, $W=\bigcup_{n\in\mathbb{N}}\Omega_n(W),$ so $\{\Omega_n(W)\}_{n\in \N}$ is an ascending filtration of $W$.

Define \[\mathrm{Gr}_n(W)=\Omega_n(W)/\Omega_{n-1}(W)\] and \[\mathrm{Gr}(W)=\displaystyle\oplus_{n\in \mathbb{N}}\mathrm{Gr}_n(W).\] Let $[w]_n$ denote $w+\Omega_{n-1}(W)\in \mathrm{Gr}_n(W),$ where $w\in\Omega_n(W)$. 

Given $\mathfrak{v}=[v]_{kl}$ and $\mathfrak{w}=[w]_n$, where $v\in V$, $w\in W$, and $k,l,n\in \mathbb{N}$, define
\begin{equation}
\vartheta_{\mathrm{Gr}(W)}([v]_{kl})[w]_n=\delta_{ln}[\Res_xx^{l-k-1}Y_W(x^{L(0)}v, x)w]_k.
\end{equation}
Extend this definition linearly to obtain the definition of $\vartheta_{\mathrm{Gr}(W)}(\mathfrak{v})\mathfrak{w}$ for arbitrary elements $\mathfrak{v}\in U^\infty(V)$ and $\mathfrak{w}\in \mathrm{Gr}(W)$.

By Theorem 2.4 in \cite{H2}, we have

\begin{theorem}
Given a grading-restricted vertex algebra $V$ and a lower-bounded generalized $V$-module, the linear map 
\begin{equation}
\vartheta_{\mathrm{Gr}(W)}:U^\infty(V)\rightarrow \End \mathrm{Gr}(W),
\end{equation}given by 
\begin{equation*}
\mathfrak{v}\mapsto \vartheta_{\mathrm{Gr}(W)}(\mathfrak{v}),
\end{equation*} gives a $U^\infty(V)$-module structure on $\mathrm{Gr}(W).$ In other words, $\vartheta_{\mathrm{Gr}(W)}$ is a homomorphism of algebras from $U^\infty(V)$ to $\End \mathrm{Gr}(W).$
\end{theorem}
Define $Q^\infty(V)$ to be the intersection of $\ker\vartheta_{\mathrm{Gr}(W)}$ for all lower-bounded generalized $V$-modules $W$ and define $A^\infty(V)=U^\infty(V)/Q^\infty(V)$. The following theorem is Theorem 2.8 in \cite{H2}:
\begin{theorem}Given a grading-restricted vertex algebra $V$, the product $\diamond$ on $U^\infty(V)$ induces a product on $A^\infty(V)$, also denoted by $\diamond$, such that $A^\infty(V)$ equipped with $\diamond$ is an associative algebra. Moreover, the associated graded space $\mathrm{Gr}(W)$ of the ascending filtration $\{\Omega_n(W)\}_{n\in \N}$ of any lower-bounded generalized $V$-module $W$ is an $A^\infty(V)$-module, where the action is the one induced from the action of $U^\infty(V)$ on $\mathrm{Gr}(W).$
\end{theorem}
\section{The Diagonal Shift Property}\label{diagonal shift}

In this section we introduce the diagonal shift property. Given a grading-restricted vertex algebra $V$, we review the subspace $O^\infty(V)$ defined by Huang, along with the Jacobi identity elements of $Q^\infty(V)$ also introduced by Huang \cite{H3,H2}. We then prove that elements of $O^\infty(V)$ and the Jacobi identity elements satisfy the diagonal shift property.

\begin{definition}Let $V$ be a grading-restricted vertex algebra. Let $v\in V$ and $k,l\in \Z_+$ such that $[v]_{kl}\in Q^\infty(V)$. We say that $[v]_{kl}$ satisfies the diagonal shift property if $[v]_{k-1,l-1}\in Q^\infty(V)$.
\end{definition}

\begin{definition}\cite{H2}
Given a grading-restricted vertex algebra $V$, define $O^\infty(V)$ to be the space spanned by infinite linear combinations of
\begin{equation}\label{circ_prod}
\Res_xx^{-k-l-p-2}(1+x)^l[Y((1+x)^{L(0)}u, x)v]_{kl}
\end{equation} for $u, v\in V$ and $k, l, p\in \mathbb{N}$ and elements of the form
\begin{equation}\label{L}
[(L(-1)+L(0)+l-k)v]_{kl}
\end{equation}for $v\in V$ and $k, l\in \mathbb{N}$, with each pair $(k,l)$ in the linear combinations appearing only finitely many times. 
\end{definition} Define $O^\infty_\circ(V)$ to be the space spanned by infinite linear combinations of
\begin{equation}\label{circ_prod}
\Res_xx^{-k-l-p-2}(1+x)^l[Y((1+x)^{L(0)}u, x)v]_{kl}
\end{equation} for $u, v\in V$ and $k, l, p\in \mathbb{N}$, with each pair $(k,l)$ in the linear combinations appearing only finitely many times. We note that $O^\infty_\circ(V)$ is analogous to the subspaces $O_n^\circ(V)\subset O_n(V)$ defined in \cite{AB1,AB2}.

The following is a consequence of Proposition 4.4 of \cite{H3} with a modification of the constraints on the integers $k,l,p,n\in \mathbb{Z}$. Since the integers allowed here are more general than those given in \cite{H3}, we include a proof. 
\begin{proposition}
Let $V$ be a grading-restricted vertex algebra. Then for $k\in \mathbb{N}$ and $l, p, n\in \mathbb{Z}$ such that $l+p\in \mathbb{N}$, and $u,v\in V$ with $v$ homogeneous, the following is in $Q^\infty(V)$. \begin{eqnarray}\label{J_elements}
\lefteqn{\sum_{\substack{j\in \mathbb{N}\\n+p-j\ge 0}}(-1)^j\binom{p}{j}[v]_{k, n+p-j}\diamond[u]_{n+p-j, l+p}}\nonumber\\&&-\sum_{\substack{j\in \mathbb{N}\\l-n+k+p-j\ge 0}}(-1)^{p-j}\binom{p}{j}[u]_{k, l-n+k+p-j}\diamond[v]_{l-n+k+p-j, l+p}\nonumber\\ &&-\sum_{j\in \mathbb{N}}\binom{\mathrm{wt}\,v+n-k-1}{j}[v_{p+j}u]_{k, l+p}.
\end{eqnarray}
\end{proposition}
\begin{proof} We may assume $u$ is homogeneous. The general result will follow by linearity. Let $W$ be a lower-bounded generalized $V$-module and let $[w]_{l+p}\in \mathrm{Gr}_{l+p}(W)$. Then \begin{eqnarray*}
\lefteqn{\vartheta_{Gr(W)}\left(\sum_{\substack{j\in \mathbb{N}\\n+p-j\ge 0}}(-1)^j\binom{p}{j}[v]_{k, n+p-j}\diamond [u]_{n+p-j, l+p}\right)[w]_{l+p}}\\&=&\sum_{\substack{j\in \mathbb{N}\\n+p-j\ge 0}}(-1)^j\binom{p}{j}\vartheta_{Gr(W)}\left([v]_{k, n+p-j}\right)\vartheta_{Gr(W)}\left([u]_{n+p-j, l+p}\right)[w]_{l+p}\\&=&\sum_{\substack{j\in \mathbb{N}\\n+p-j\ge 0}}(-1)^j\binom{p}{j}\vartheta_{Gr(W)}\left([v]_{k, n+p-j}\right)[\Res_xx^{l+p-(n+p-j)-1}Y_W(x^{L(0)}u, x)w]_{n+p-j}\\&=&\sum_{\substack{j\in \mathbb{N}\\n+p-j\ge 0}}(-1)^j\binom{p}{j}\vartheta_{Gr(W)}\left([v]_{k, n+p-j}\right)[u_{\text{wt }u-1+l-n+j}w]_{n+p-j}\\&=&\sum_{\substack{j\in \mathbb{N}\\n+p-j\ge 0}}(-1)^j\binom{p}{j}[\Res_xx^{n+p-j-k-1}Y_W(x^{L(0)}v, x)u_{\text{wt }u-1+l-n+j}w]_k\\&=&\sum_{\substack{j\in \mathbb{N}\\n+p-j\ge 0}}(-1)^j\binom{p}{j}[v_{\text{wt }v-1+n+p-j-k}u_{\text{wt }u-1+l-n+j}w]_k.
\end{eqnarray*}
Next observe that since $[w]_{l+p}\in \mathrm{Gr}_{l+p}(W)=\Omega_{l+p}/\Omega_{l+p-1}(W)$, $[w]_{l+p}=w+\Omega_{l+p-1}(W)$, where $w\in \Omega_{l+p}(W)=\{w\in W|v_kw=0\text{ for homogeneous }v\in V\text{ and }{\text{wt }v-k-1<-l-p}\}$. Note that since $n-l-j=n-j+p-l-p$, if $n-j+p<0$, then $n-l-j<-l-p$. But $n-l-j$ is the weight of $u_{\text{wt }u-1+l-n+j}$ as an operator. Therefore, $u_{\text{wt }u-1+l-n+j}w=0$ when $n-j+p<0$ and the right hand side of the above expression is equal to 
\begin{eqnarray}\label{Jacobi}
\lefteqn{\sum_{{j\in \mathbb{N}}}(-1)^j\binom{p}{j}[v_{\text{wt }v-1+n+p-j-k}u_{\text{wt }u-1+l-n+j}w]_k}\nonumber\\&=&\sum_{j\in \mathbb{N}} (-1)^{p-j}\binom{p}{j}[u_{\text{wt }u-1+l-n+p-j}v_{\text{wt }v-1+n-k+j}w]_k\nonumber\\&+&\sum_{j\in \mathbb{N}}\binom{\text{wt }v+n-k-1}{j}[(v_{p+j}u)_{\text{wt }u+\text{wt }v+l-k-2-j}w]_k,
\end{eqnarray}where the equality in Equation \eqref{Jacobi} follows from the Jacobi identity for modules.

If $l-n+k+p-j<0$, then $\text{wt }v-(\text{wt }v-1+n-k+j)-1=k-n-j=l-n+k+p-j-l-p<-l-p,$ so in this case, $v_{\text{wt }v-1+n-k+j}w=0$ and therefore the first sum on the right hand side of Equation \eqref{Jacobi} is over $j\in \mathbb{N}$ such that $l-n+k+p-j\ge 0$. Thus,
\begin{eqnarray*}
\lefteqn{\vartheta_{\mathrm{Gr}(W)}\left(\sum_{\substack{j\in \mathbb{N}\\n+p-j\ge 0}}(-1)^j\binom{p}{j}[v]_{k, n+p-j}\diamond [u]_{n+p-j, l+p}\right)[w]_{l+p}}\\&-&\sum_{\substack{j\in \mathbb{N}\\l-n+k+p-j\ge 0}} (-1)^{p-j}\binom{p}{j}[u_{\text{wt }u-1+l-n+p-j}v_{\text{wt }v-1+n-k+j}w]_k\\&-&\sum_{j\in \mathbb{N}}\binom{\text{wt }v+n-k-1}{j}[(v_{p+j}u)_{\text{wt }u+\text{wt }v+l-k-2-j}w]_k=0,
\end{eqnarray*}or equivalently,
\begin{eqnarray*}
\lefteqn{\vartheta_{\mathrm{Gr}(W)}\left(\sum_{\substack{j\in \mathbb{N}\\n+p-j\ge 0}}(-1)^j\binom{p}{j}[v]_{k, n+p-j}\diamond [u]_{n+p-j, l+p}\right)[w]_{l+p}}\\&-&\vartheta_{\mathrm{Gr}(W)}\left(\sum_{\substack{j\in \mathbb{N}\\l-n+k+p-j\ge 0}} (-1)^{p-j}\binom{p}{j}[u]_{k,l+p-n+k-j}\diamond[v]_{l+p-n+k-j,l+p}\right)[w]_{l+p}\\&-&\vartheta_{\mathrm{Gr}(W)}\left(\sum_{j\in \mathbb{N}}\binom{\text{wt }v+n-k-1}{j}[v_{p+j}u]_{k,l+p}\right)[w]_{l+p}=0.
\end{eqnarray*}
\end{proof}

\begin{definition}Given a grading restricted vertex algebra, define $J^\infty(V)$ to be the subspace of $U^\infty(V)$ generated by elements of the form \eqref{J_elements}.
\end{definition}

\begin{proposition}\label{gen_Huang_subset}Let $V$ be a grading-restricted vertex algebra and let $k,k',l,l'\in \mathbb{N}$ be such that $l'\ge l$ and $k'\ge k$. If $[v]_{k'l'}\in O^{\infty}_\circ(V)$ then $[v]_{kl}\in O^{\infty}_\circ(V)$.
\end{proposition}

\begin{proof}If $a\in V$ such that $[a]_{k'l'}\in O^{\infty}_\circ(V)$, then $a$ is a linear combination of elements of the form \[\Res_xx^{-k'-l'-p-2}(1+x)^{l'}Y((1+x)^{L(0)}u, x)v,\] where $u,v\in V$ and $p\in \mathbb{N}$. Therefore, we need only show that given $u,v\in V$ and $p\in \mathbb{N}$, \[\Res_xx^{-k'-l'-p-2}(1+x)^{l'}[Y((1+x)^{L(0)}u, x)v]_{kl}\in O^{\infty}_\circ(V).\]

We have
\begin{eqnarray*}
\lefteqn{\Res_xx^{-k'-l'-p-2}(1+x)^{l'}[Y((1+x)^{L(0)}u, x)v]_{kl}}\\&&=\Res_xx^{-(k'+k-k)-(l'+l-l)-p-2}(1+x)^{l'+l-l}[Y((1+x)^{L(0)}u, x)v]_{kl}\\&&=\Res_xx^{-k-l-(p+k'-k+l'-l)-2}\sum_{i=0}^{l'-l}\binom{l'-l}{i}x^i(1+x)^{l}[Y((1+x)^{L(0)}u, x)v]_{kl}\\&&=\sum_{i=0}^{l'-l}\binom{l'-l}{i}\Res_xx^{-k-l-(p+k'-k+l'-l-i)-2}(1+x)^{l}[Y((1+x)^{L(0)}u, x)v]_{kl}.
\end{eqnarray*}Since $k'\ge k$, $l'\ge l$, and $p\in \mathbb{N}$, we have that $p+k'-k+l'-l-i\ge 0$ for all $i$ such that $0\le i\le l'-l$. The result then follows from the definition of $O^{\infty}_\circ(V)$.
\end{proof} An immediate consequence of Proposition \ref{gen_Huang_subset} is the following corollary: \begin{corollary}\label{kl}
    Let $V$ be a VOA and let $k,k',l,l'\in \mathbb{N}$ such that $k'\ge k$ and $k-l=k'-l'$.  If $[v]_{k'l'}\in O^{\infty}(V)$ then $[v]_{kl}\in O^{\infty}(V)$. 
\end{corollary}
\begin{proof}Proposition \ref{gen_Huang_subset} implies that if $[v]_{k'l'}\in O^\infty_\circ(V)$ then $[v]_{kl}\in O^\infty_\circ(V).$ It is an obvious observation that for $v\in V$, since $k-l=k'-l'$, then $[(L(-1)+L(0)+l'-k')v]_{kl}\in O^\infty(V).$ The result follows.
\end{proof}
\begin{corollary}
If $V$ is a grading-restricted vertex algebra and $[v]_{kl}\in O^\infty(V)$ such that $k,l\in \Z_+$ then $[v]_{k-1,l-1}\in O^\infty(V).$ In particular, elements of $O^\infty(V)$ satisfy the diagonal shift property.
\end{corollary}

We next prove that given a grading-restricted vertex algebra $V$, elements of $J^\infty(V)$ satisfy the diagonal shift property. In order to do this, we will need an analogue of Proposition \ref{mult_equiv} for the product $\diamond$ used to define $A^\infty(V):$

\begin{proposition}\label{mult}Let $V$ be a grading-restricted vertex algebra and let $u,v\in V$, $i, k,l\in \Z_+$. Define\begin{equation}\label{diamond}
q:=\sum_{m=0}^i\binom{-k+i-l-1}{m}\Res_xx^{-k+i-l-m-1}(1+x)^lY((1+x)^{L(0)}u, x)v, 
\end{equation}so that by the definition of the product $\diamond$, we have
\begin{equation}[u]_{ki}\diamond [v]_{il}=[q]_{kl}.
\end{equation}Then
$[u]_{k-1,i-1}\diamond [v]_{i-1,l-1}\equiv[q]_{k-1,l-1}$ modulo $O^\infty_\circ(V).$
\end{proposition}

\begin{proof}The proof is analogous to the proof of Proposition 2.4 in \cite{DLM}. For simplicity, we'll take $u$ to be of homogeneous weight. The general result then follows by linearity. We have
\begin{eqnarray*}
[q]_{k-1,l-1}&=&[\sum_{m=0}^i\binom{-k+i-l-1}{m}\Res_x\frac{(1+x)^{l+\mathrm{wt} u}}{x^{k-i+l+m+1}}Y(u, x)v]_{k-1,l-1}\\&=&[\sum_{m=0}^i\binom{-k+i-l-1}{m}\Res_x\frac{(1+x)^{l+\mathrm{wt} u-1}}{x^{k-i+l+m}}Y(u, x)v]_{k-1,l-1}+\\&&[\sum_{m=0}^i\binom{-k+i-l-1}{m}\Res_x\frac{(1+x)^{l+\mathrm{wt} u-1}}{x^{k-i+l+m+1}}Y(u, x)v]_{k-1,l-1}\\&\equiv&[\sum_{m=0}^{i-1}\binom{-k+i-l-1}{m}\Res_x\frac{(1+x)^{l+\mathrm{wt} u-1}}{x^{k-i+l+m}}Y(u, x)v]_{k-1,l-1}\\&&+[\sum_{m=0}^{i-2}\binom{-k+i-l-1}{m}\Res_x\frac{(1+x)^{l+\mathrm{wt} u-1}}{x^{k-i+l+m+1}}Y(u, x)v]_{k-1,l-1}\mathrm{\, mod\, }O^\infty_\circ(V)\\&=&\Res_x[\frac{Y(u,x)v(1+x)^{\mathrm{wt}u+l-1}}{x^{k+l-i}}]_{k-1,l-1}\\&&+\sum_{m=1}^{i-1}\left(\binom{-k+i-l-1}{m}+\binom{-k+i-l-1}{m-1}\right)[\Res_x\frac{Y(u,x)v(1+x)^{\mathrm{wt}u+l-1}}{x^{k+l-i+m}}]_{k-1,l-1}\\&=&[\Res_x\sum_{m=0}^{i-1}\binom{-k+i-l}{m}\frac{Y(u,x)v(1+x)^{\mathrm{wt}u+l-1}}{x^{k+l-i+m}}]_{k-1,l-1}\\&=&[u]_{k-1,i-1}\diamond[v]_{i-1,l-1}.
\end{eqnarray*}
\end{proof}

\begin{proposition}\label{mult0}Let $k,l\in \Z_+$ and $u,v\in V$. Define 
\begin{equation*}r:=\Res_x\frac{(1+x)^lY((1+x)^{L(0)}u,x)v}{x^{k+l+1}}
\end{equation*}so that
\begin{equation*}
[r]_{kl}=[u]_{k0}\diamond [v]_{0l}.
\end{equation*}Then $[r]_{k-1,l-1}\in O^\infty_\circ(V).$
\end{proposition}
\begin{proof}
Let $u$ be homogeneous. The general result then follows by linearity. We have
\begin{eqnarray*}
[r]_{k-1,l-1}&=&[\Res_x \frac{(1+x)^{\mathrm{wt}u+l}}{x^{k+l+1}}Y(u,x)v]_{k-1,l-1}=[\Res_x \frac{(1+x)^{\mathrm{wt}u+l-1}(1+x)}{x^{(k-1)+(l-1)+3}}Y(u,x)v]_{k-1,l-1}\\&=&[\Res_x \frac{(1+x)^{\mathrm{wt}u+l-1}}{x^{(k-1)+(l-1)+3}}Y(u,x)v]_{k-1,l-1}+[\Res_x \frac{(1+x)^{\mathrm{wt}u+l-1}}{x^{(k-1)+(l-1)+2}}Y(u,x)v]_{k-1,l-1}\in O^\infty_\circ(V).
\end{eqnarray*}
\end{proof}
\begin{proposition}\label{J}Let $V$ be a grading-restricted vertex algebra and let $u,v\in V$ with $v$ homogeneous. Let $k\in \mathbb{N}$ and $l, p, n\in \mathbb{Z}$ such that $k, l+p\in \Z_{+}$. Define 
\begin{eqnarray*}
s&:=&\sum_{\substack{j\in \mathbb{N}\\n+p-j\ge 0}}(-1)^j\binom{p}{j}\sum_{m=0}^{n+p-j}\binom{-k+n-j-l-1}{m}\Res_x\frac{(1+x)^{l+p}}{x^{k-n+j+l+m+1}}Y((1+x)^{L(0)}v,x)u\\&&-\sum_{\substack{j\in \mathbb{N}\\l-n+k+p-j\ge 0}}(-1)^{p-j}\binom{p}{j}\sum_{m=0}^{l-n+k+p-j}\binom{-n-j-1}{m}\Res_x\frac{(1+x)^{l+p}}{x^{n+j+m+1}}Y((1+x)^{L(0)}u,x)v\\ &&-\sum_{j\in \mathbb{N}}\binom{\mathrm{wt}\,v+n-k-1}{j}v_{p+j}u,
\end{eqnarray*} so that in particular, 
\begin{eqnarray}\label{J1}[s]_{k,l+p}&=&
\sum_{\substack{j\in \mathbb{N}\\n+p-j\ge 0}}(-1)^j\binom{p}{j}[v]_{k, n+p-j}\diamond[u]_{n+p-j, l+p}\nonumber\\&&-\sum_{\substack{j\in \mathbb{N}\nonumber\\l-n+k+p-j\ge 0}}(-1)^{p-j}\binom{p}{j}[u]_{k, l-n+k+p-j}\diamond[v]_{l-n+k+p-j, l+p}\nonumber\\ &&-\sum_{j\in \mathbb{N}}\binom{\mathrm{wt}\,v+n-k-1}{j}[v_{p+j}u]_{k, l+p}
\in Q^\infty(V).
\end{eqnarray} Then $[s]_{k-1,l+p-1}\in Q^\infty(V).$ In particular, elements of $J^\infty(V)$ satisfy the diagonal shift property.
\end{proposition}

\begin{proof}Applying Proposition \ref{mult}, we have \begin{eqnarray*}[s]_{k-1,l+p-1}&\equiv& \sum_{\substack{j\in \N\\j=n+p}}(-1)^{n+p}\binom{p}{n+p}[\Res_x\frac{(1+x)^{l+p}Y((1+x)^{L(0)}v,x)u}{x^{k+l+p+1}}]_{k-1,l+p-1}\\&&+\lefteqn{\sum_{\substack{j\in \mathbb{N}\\n+p-j-1\ge 0}}(-1)^j\binom{p}{j}[v]_{k-1, n+p-j-1}\diamond[u]_{n+p-j-1, l+p-1}}\\&&-\sum_{\substack{j\in \N\\j=l-n+k+p}}(-1)^{-l+n-k}\binom{p}{l-n+k+p}[\Res_x\frac{(1+x)^{l+p}Y((1+x)^{L(0)}u,x)v}{x^{k+l+p+1}}]_{k-1,l+p-1}\\&&-\sum_{\substack{j\in \mathbb{N}\\l-n+k+p-j-1\ge 0}}(-1)^{p-j}\binom{p}{j}[u]_{k-1, l-n+k+p-j-1}\diamond[v]_{l-n+k+p-j-1, l+p-1}\\ &&-\sum_{j\in \mathbb{N}}\binom{\text{wt}\,v+n-k-1}{j}[v_{p+j}u]_{k-1, l+p-1} \mathrm{ \,mod \,}O^\infty_\circ(V).
\end{eqnarray*} By Proposition \ref{mult0}, the first and third sums on the right hand side of this expression are equivalent modulo $O^\infty_\circ(V)$ to $0$. Thus the right hand side of the expression is equivalent modulo $Q^\infty(V)$ to 
\begin{eqnarray*}
&&\sum_{\substack{j\in \mathbb{N}\\n+p-j-1\ge 0}}(-1)^j\binom{p}{j}[v]_{k-1, n+p-j-1}\diamond[u]_{n+p-j-1, l+p-1}\nonumber\\&&-\sum_{\substack{j\in \mathbb{N}\nonumber\\l-n+k+p-j-1\ge 0}}(-1)^{p-j}\binom{p}{j}[u]_{k-1, l-n+k+p-j-1}\diamond[v]_{l-n+k+p-j-1, l+p-1}\nonumber\\&&-\sum_{j\in \mathbb{N}}\binom{\text{wt}\,v+n-k-1}{j}[v_{p+j}u]_{k-1, l+p-1},
\end{eqnarray*}
which is an expression of the form \eqref{J_elements}, obtained by sending $k\mapsto k-1,$ $l\mapsto l-1,$ and $n\mapsto n-1.$  Thus $[s]_{k-1,l+p-1}\in Q^\infty(V)$, completing the proof.
\end{proof}

We summarize the main results of this paper in the following theorem.

\begin{theorem}\label{main}Let $V$ be a grading restricted vertex algebra. Then elements generated by $O^\infty(V)$ and $J^\infty(V)$ satisfy the diagonal shift property.
\end{theorem}

\section{The Rank One Heisenberg VOA}\label{Heis}

In this section we show that in the case that $V$ is the rank one Heisenberg VOA, there exist elements in $Q^\infty(V)$ that do not satisfy the diagonal shift property.

Let $\mathfrak{h}$ be a one-dimensional Lie algebra with a bilinear form $(\cdot, \cdot)$. Let $\alpha\in \mathfrak{h}$ such that $(\alpha, \alpha)=1$. Define $\widehat{\mathfrak{h}}$ to be the affinization of $\mathfrak{h}$,
\[\widehat{\mathfrak{h}}=\alpha\otimes \mathbb{C}[t, t^{-1}]\oplus \mathbb{C}c\] with Lie bracket given by\begin{align*}[\alpha(m), \alpha(n)]&=m\delta_{m+n,0}c\\ [\alpha(m), c]&=0,\end{align*}where $\alpha(m)=\alpha\otimes t^m$.

Define $\widehat{\mathfrak{h}}^-=\mathfrak{h}\otimes t^{-1}\mathbb{C}[t^{-1}]$ and $\widehat{\mathfrak{h}}^+=\mathfrak{h}\otimes t\mathbb{C}[t]$ and let $M(1)$ be the induced $\widehat{\mathfrak{h}}$-module,
\[
M(1) = U(\widehat{\mathfrak{h}})\otimes_{U(\C[t]\otimes \mathfrak{h} \oplus \C c)} \C{\bf 1} \simeq S(\widehat{\mathfrak{h}}^{-}) \qquad \mbox{(linearly)},
\]
where $U(\cdot)$ and $S(\cdot)$ denote the universal enveloping algebra and symmetric algebra, respectively, $\mathfrak{h} \otimes \C[t]$ acts trivially on $\mathbb{C}\mathbf{1},$ and $c$ acts as multiplication by $1$. Then $M(1)$ is a VOA called the {\emph{rank one Heisenberg VOA}}.

The following proposition characterizes lower-bounded generalized $M(1)$-modules.
\begin{proposition}\cite{FLM, LL, M}\label{hmod} Let $W$ be a lower-bounded generalized $M(1)$-module. Then as an $\widehat{\mathfrak{h}}$-module, 
\[W\cong M(1)\otimes \Omega(W),\] where $\Omega(W)=\{w\in W|\alpha(n)w=0, n>0\}.$
\end{proposition}
\begin{proof}Use the canonical $\mathbb{N}$-grading of $W$ to prove that $W$ is a restricted $\mathbb{Z}$-graded $\widehat{\mathfrak{h}}$-module. The result then follows using the same argument as in the proof of Proposition 3.1 in \cite{M}.
\end{proof}

The next proposition provides an infinite set of linearly independent vectors in $Q^\infty(M(1))$ that do not satisfy the diagonal shift property.

\begin{proposition}\label{ker1}Fix $n\in \mathbb{\Z}_+$. Then $[\alpha(-1){\bf 1}]_{nn}\diamond[\alpha(-1){\bf 1}]_{nn}-[\alpha(-1)^2{\bf 1}]_{nn}+2n[{\bf 1}]_{nn}$ is an element of  $Q^\infty(M(1))$ that does not satisfy the diagonal shift property. 
\end{proposition}
\begin{proof}
Let $W$ be a lower-bounded generalized $M(1)$-module. Then by Proposition \ref{hmod}, \begin{equation*}W\cong M(1)\otimes \Omega(W).\end{equation*} Thus for $n\in \N$ with $n>0,$ \begin{equation*}\Omega_{n}(W)=\text{span }\{\alpha(-i_1)\cdots\alpha(-i_j){\bf 1}\otimes w|w\in \Omega(W),\, j\in \N,\, i_1+\cdots+i_j\le n\}\end{equation*} and \begin{equation*}\Omega_{n-1}(W)=\text{span }\{\alpha(-i_1)\cdots\alpha(-i_j){\bf 1}\otimes w|w\in \Omega(W),\, j\in \N,\, i_1+\cdots+i_j\le n-1\}\end{equation*} so that
\begin{eqnarray*}\mathrm{Gr}_n(W)&=&\Omega_n(W)/\Omega_{n-1}(W)\\&=&\text{span }\{\alpha(-i_1)\cdots\alpha(-i_j){\bf 1}\otimes w+\Omega_{n-1}(W)|w\in \Omega(W),\,j\in \N,\,i_1+\cdots+i_j=n\}.\end{eqnarray*} 

Let $x=\alpha(-i_1)\cdots\alpha(-i_j){\bf 1}\otimes w$ where $w\in \Omega(W)$ and $i_1+\cdots+i_j=n.$ 
Then
\begin{eqnarray*}\lefteqn{\vartheta_{\mathrm{Gr}(W)}([\alpha(-1){\bf 1}]_{nn}\diamond[\alpha(-1){\bf 1}]_{nn}-[\alpha(-1)^2{\bf 1}]_{nn}+2n[{\bf 1}]_{nn})[x]_n}\\&=&[(\alpha(0)^2-(\alpha(0)^2+2\sum_{k=1}^\infty\alpha(-k)\alpha(k))+2n)\alpha(-i_1)\cdots\alpha(-i_j)\otimes w]_{n}\\&=&[(2n-\sum_{k=1}^\infty\alpha(-k)\alpha(k))\alpha(-i_1)\cdots\alpha(-i_j)\otimes w]_{n}=[0]_{n},
\end{eqnarray*} so $[\alpha(-1){\bf 1}]_{nn}\diamond[\alpha(-1){\bf 1}]_{nn}-[\alpha(-1)^2{\bf 1}]_{nn}+2n[{\bf 1}]_{nn}\in Q^\infty(M(1))$. 

Next, recall that by Proposition \ref{mult}, if \begin{equation}\label{diamond}
y:=\sum_{m=0}^n\binom{-n-1}{m}\Res_xx^{-n-m-1}(1+x)^nY((1+x)^{L(0)}\alpha(-1){\bf 1}, x)\alpha(-1){\bf 1}, 
\end{equation}so that by the definition of the product $\diamond$, we have
\begin{equation*}[\alpha(-1){\bf 1}]_{nn}\diamond [\alpha(-1){\bf 1}]_{nn}=[y]_{nn},\end{equation*} then $[\alpha(-1){\bf 1}]_{n-1,n-1}\diamond [\alpha(-1){\bf 1}]_{n-1,n-1}\equiv[y]_{n-1,n-1}$ modulo $Q^\infty(M(1))$. Thus, in order to show that $[\alpha(-1){\bf 1}]_{nn}\diamond[\alpha(-1){\bf 1}]_{nn}-[\alpha(-1)^2{\bf 1}]_{nn}+2n[{\bf 1}]_{nn}$ does not satisfy the diagonal shift property, we need only show that $[\alpha(-1){\bf 1}]_{n-1,n-1}\diamond[\alpha(-1){\bf 1}]_{n-1,n-1}-[\alpha(-1)^2{\bf 1}]_{n-1,n-1}+2n[{\bf 1}]_{n-1,n-1}\notin Q^\infty(M(1))$.

Let $x={\alpha(-1)^{n-1}\bf 1}\otimes w$ where $w\in \Omega(W)$, then
\begin{eqnarray*}
&&\vartheta_{Gr(W)}([\alpha(-1){\bf 1}]_{n-1,n-1}\diamond[\alpha(-1){\bf 1}]_{n-1,n-1}-[\alpha(-1)^2{\bf 1}]_{n-1,n-1}+2n[{\bf 1}]_{n-1,n-1})[x]_{n-1}\\&=&[(\alpha(0)^2-(\alpha(0)^2+2\sum_{k=1}^\infty\alpha(-k)\alpha(k))+2n)\alpha(-1)^{n-1}{\bf 1}\otimes w]_{n-1}=2[x]_{n-1}\neq[0]_{n-1}.
\end{eqnarray*}Thus $[\alpha(-1){\bf 1}]_{n-1,n-1}\diamond[\alpha(-1){\bf 1}]_{n-1,n-1}-[\alpha(-1)^2{\bf 1}]_{n-1,n-1}+2n[{\bf 1}]_{n-1,n-1}\notin Q^\infty(M(1))$ and therefore $[\alpha(-1){\bf 1}]_{nn}\diamond[\alpha(-1){\bf 1}]_{nn}-[\alpha(-1)^2{\bf 1}]_{nn}+2n[{\bf 1}]_{nn}$ does not satisfy the diagonal shift property.
\end{proof}

\begin{remark} From Theorem 6 of \cite{BVY_Heis}, one has

\begin{theorem}\label{BVY_Heis}
Let $I$ be the ideal generated by $(x^2-y)(x^2-y+2)$ in $\mathbb{C}[x, y]$. Then
\begin{eqnarray*}A_1(M(1))\cong \mathbb{C}[x, y]/I
\end{eqnarray*}under the identifications 
\[\alpha(-1){\bf 1}+O_1(M(1))=x+I\quad \quad \alpha(-1)^2{\bf 1}+O_1(M(1))=y+I.\]
\end{theorem}

To motivate the construction of \[[\alpha(-1){\bf 1}]_{11}\diamond [\alpha(-1){\bf 1}]_{11}-[\alpha(-1)^2{\bf 1}]_{11}+2[{\bf 1}]_{11}=[\alpha(-1){\bf 1}*_1\alpha(-1){\bf 1}-\alpha(-1)^2{\bf 1}+2{\bf 1}]_{11},\] observe that, given an admissible $M(1)$-module $M$, the factor $x^2-y$ given above in the statement of Theorem \ref{BVY_Heis} acts as $0$ on $\Omega_0(M)$ but does not act as zero on $\Omega_1(M)\setminus\Omega_0(M)$. On the other hand, the factor $x^2-y+2$ acts as $0$ on $\Omega_1(M)\setminus\Omega_0(M)$ but not as zero on $\Omega_0(M).$ One can then easily generalize the construction of $[\alpha(-1)]_{11}\diamond [\alpha(-1){\bf 1}]_{11}-[\alpha(-1)^2{\bf 1}]_{11}+2[{\bf 1}]_{11}$ to obtain the vectors given in Proposition \ref{ker1}.
\end{remark}

The following is an immediate corollary of Proposition \ref{ker1} and Theorem \ref{main}:
\begin{corollary}
    There are elements in $Q^\infty(M(1))$ that are not generated by elements of $O^\infty(M(1))$ and $J^\infty(M(1)).$
\end{corollary}

Another corollary is then:
\begin{corollary}Huang's conjecture that for an arbitrary grading restricted vertex algebra $V$,
$O^\infty(V)$ and $J^\infty(V)$ generate $Q^\infty(V)$ does not hold in general.
\end{corollary}


\begin{thebibliography}{[CCCC]}


 \bibitem[AB1]{AB1} Addabbo, D. \& Barron, K. On generators and relations for higher level Zhu algebras and applications. {\em J. Algebra}. \textbf{623} pp. 496-540 (2023)


\bibitem[AB2]{AB2} Addabbo, D. \& Barron, K. The level two Zhu algebra for the Heisenberg VOA. {\em Comm. Algebra}. \textbf{51}, 3405-3463 (2023)


\bibitem[BVY1]{BVY_gen}Barron, K., Vander Werf, N. \& Yang, J. Higher level Zhu algebras and modules for vertex operator algebras. {\em J. Pure Appl. Algebra}. \textbf{223}, 3295-3317 (2019)
 
\bibitem[BVY2]{BVY_Heis}Barron, K., Vander Werf, N. \& Yang, J. The level one Zhu algebra for the Heisenberg VOA. {\em Affine, Vertex And W-algebras}. \textbf{37} pp. 37-64

\bibitem[BVY3]{BVY_Vir}Barron, K., Vander Werf, N. \& Yang, J. The level one Zhu algebra for the Virasoro vertex operator algebra. {\em Vertex Operator Algebras, Number Theory And Related Topics}. \textbf{753} pp. 17-43
 
\bibitem[DLM]{DLM}Dong, C., Li, H. \& Mason, G. VOAs and associative algebras. {\em J. Algebra}. \textbf{206}, 67-96 (1998)

\bibitem[FHL]{FHL}Frenkel, I., Huang, Y. \& Lepowsky, J. On axiomatic approaches to vertex operator algebras and modules. {\em Mem. Amer. Math. Soc.}. \textbf{104}, viii+64 (1993)
 
 \bibitem[FLM]{FLM}Frenkel, I., Lepowsky J. and Meurman, A. {\it VOAs and the Monster}, Pure and Appl. Math., 134, Academic Press, New York (1988)
 
 \bibitem[FZ]{FZ}Frenkel, I., Zhu, Y.-C.: {\it VOAs associated to representations of affine and Virasoro algebras}. Duke Math. J. 66, 123–168 (1992)

\bibitem[H]{H}Huang, Y. Differential equations, duality and modular invariance. {\em Commun. Contemp. Math.}. \textbf{7}, 649-706 (2005)

\bibitem[H1]{H4}Huang, Y. Cofiniteness conditions, projective covers and the logarithmic tensor product theory. {\em J. Pure Appl. Algebra}. \textbf{213}, 458-475 (2009)

\bibitem[H2]{H1}Huang, Y. A cohomology theory of grading-restricted vertex algebras. {\em Comm. Math. Phys.}. \textbf{327}, 279-307 (2014)

\bibitem[H3]{H3}Huang, Y. Associative algebras and intertwining operators. {\em Comm. Math. Phys.}. \textbf{396}, 1-44 (2022)
 
\bibitem[H4]{H2}Huang, Y. Associative algebras and the representation theory of grading-restricted vertex algebras. {\em Commun. Contemp. Math.}. \textbf{26}, Paper No. 2350036, 46 (2024)
 
 
\bibitem[H5]{H5} Huang, Y. Modular invariance of (logarithmic) intertwining operators. {\em Comm. Math. Phys.}. \textbf{405}, Paper No. 131, 82 (2024)

 \bibitem[LL]{LL}Lepowsky, J. \& Li, H. Introduction to VOAs and their representations. (Birkhäuser Boston, Inc., Boston, MA, 2004)

\bibitem[M]{M}Milas, A. Logarithmic intertwining operators and vertex operators. {\em Comm. Math. Phys.}. \textbf{277}, 497-529 (2008)

\bibitem[Miy]{Miy}Miyamoto, M. Modular invariance of vertex operator algebras satisfying $C_2$-cofiniteness. {\em Duke Math. J.}. \textbf{122}, 51-91 (2004)

\bibitem[Z]{Z}Zhu, Y. Modular invariance of characters of VOAs. {\em J. Amer. Math. Soc.}. \textbf{9}, 237-302 (1996)
\end{thebibliography}
\end{document}